\documentclass[12pt]{amsart}
\usepackage[utf8]{inputenc}
\usepackage[margin=1in]{geometry}
\usepackage[dvipsnames]{xcolor}
\usepackage{bbm}
\usepackage{tikz}
\usetikzlibrary{arrows}
\usepackage{tikz-cd}
\usepackage{pgffor}
\usepackage{amssymb,amsmath,amsfonts,mathrsfs,amsthm}
\usepackage{enumitem}
\usepackage{adjustbox}
\usepackage{hyperref}
\usepackage{relsize}
\usepackage{graphicx}
\usepackage[font={small}]{caption}
\usepackage{float}
\usepackage{subcaption}
\usepackage{changepage}
\usepackage{changes}
\usepackage[toc]{appendix}
\usepackage[mathscr]{eucal}
\usepackage{pgfplots}
\pgfplotsset{compat=1.7}
\usepackage{adjustbox}
\usepackage{varwidth}
\usepackage{chngcntr}

\usepackage{tcolorbox}
\usepackage{mathtools}

\newcommand{\R}{\mathbb{R}}
\newcommand{\Q}{\mathbb{Q}}
\newcommand{\Z}{\mathbb{Z}}

\renewcommand{\l}{\ell}

\renewcommand{\sl}{\mathfrak{sl}}

\newcommand{\Cs}{\mathscr{C}}

\renewcommand{\d}{\partial}

\newcommand{\Zq}{\Z[q,q^{-1}]}

\newcommand{\BN}{\mathscr{BN}}
\newcommand{\BBN}{\mathscr{BBN}}

\newcommand{\A}{\mathbb{A}}

\newcommand{\F}{\mathcal{F}}
\newcommand{\til}[1]{\widetilde{#1}}
\renewcommand{\b}[1]{\overline{#1}}

\newcommand{\mmod}{{\rm -}\operatorname{mod}}
\newcommand{\gmod}{{\rm -}\operatorname{gmod}}
\newcommand{\ggmod}{{\rm -}\operatorname{ggmod}}
\renewcommand{\k}{\Bbbk}
\newcommand{\qdeg}{\operatorname{qdeg}}
\newcommand{\adeg}{\operatorname{adeg}}
\newcommand{\lr}[1]{\vert {#1} \vert}

\newcommand{\rar}[1]{\xrightarrow{#1}}

\renewcommand{\v}{v}
\newcommand{\lrar}{\leftrightarrow}

\renewcommand{\u}[1]{\underline{#1}}
\newcommand{\al}{\alpha}
\newcommand{\D}{\mathcal{D}}
\newcommand{\ald}{{\alpha\mathcal{D}}}
\renewcommand{\o}{\otimes}
\newcommand{\eps}{\varepsilon}

\newcommand{\Fn}{\mathcal{F}}
\newcommand{\Fa}{\mathcal{F}_{\alpha}}
\newcommand{\Fad}{\mathcal{F}_{\alpha \mathcal{D}}}
\newcommand{\Gn}{\mathcal{G}}
\newcommand{\Ga}{\mathcal{G}_{\alpha}}

\newcommand{\Gad}{\mathcal{G}_{\alpha \mathcal{D}}}
\newcommand{\ab}{\mathbf{a}}
\newcommand{\bb}{\mathbf{b}}
\newcommand{\s}{\mathfrak{s}}

\newcommand{\be}{\beta}
\renewcommand{\:}{\colon}
\newcommand{\TL}{\mathscr{TL}}

\newcommand\Label[1]{&\refstepcounter{equation}(\theequation)\ltx@label{#1}&}

\newtheorem{theorem}{Theorem}[section]
\newtheorem{lemma}[theorem]{Lemma}
\newtheorem{proposition}[theorem]{Proposition}
\newtheorem{corollary}[theorem]{Corollary}

\theoremstyle{definition}
\newtheorem{innercustomthm}{Question}
\newenvironment{question}[1]
  {\renewcommand\theinnercustomthm{#1}\innercustomthm}
  {\endinnercustomthm}
  
\theoremstyle{definition}

\theoremstyle{remark}

\theoremstyle{remark}
\newtheorem{remark}[theorem]{Remark}

\theoremstyle{definition}
\newtheorem{definition}{Definition}[section]

\theoremstyle{remark}
\newtheorem{rmk/}{Remark}

\theoremstyle{definition}
\newtheorem{case/}{Case}

\begin{document}

\title{Equivariant annular Khovanov homology}

\author[R. Akhmechet]{Rostislav Akhmechet}
\address{Department of Mathematics, University of Virginia, Charlottesville VA 22904-4137}
\email{\href{mailto:ra5aq@virginia.edu}{ra5aq@virginia.edu}}

\maketitle

\begin{abstract}
We construct an equivariant version of annular Khovanov homology via the Frobenius algebra associated with $U(1) \times U(1)$-equivariant cohomology of $\mathbb{CP}^1$. Motivated by the relationship between the Temperley-Lieb algebra and annular Khovanov homology, we also introduce an equivariant analogue of the Temperley-Lieb algebra. 
\end{abstract}

\tableofcontents

\section{Introduction}

In \cite{Kh1} Khovanov introduced a categorification of the Jones polynomial by assigning a chain complex $CKh(D)$ of graded abelian groups to a diagram $D$ of an oriented link $L\subset \R^3$. Reidemeister moves between link diagrams induce chain homotopy equivalences between the chain complexes, and the graded Euler characteristic of $CKh(D)$ is the Jones polynomial of $L$. The chain complex is built by forming the so-called \emph{cube of resolutions} and applying the two-dimensional TQFT corresponding to the Frobenius algebra 
\[
H^*(S^2;\Z) = \Z[X]/(X^2).
\]
A crossingless diagram $D$ is assigned a chain complex supported in homological degree zero by applying the TQFT directly to $D$. In particular, the empty link is assigned $H^*(\{*\};\Z)=\Z$, while the unknot is assigned $\Z[X]/(X^2)$. 

Varying the TQFT has been explored extensively, \cite{BN2, Kh3, Le, KR}, and has proven to be fruitful for topological applications, \cite{Ra}. Of particular interest is the \emph{equivariant} or \emph{universal} theory, built using the Frobenius algebra 
\[
A = \Z[E_1, E_2, X]/(X^2 - E_1 X + E_2)
\]
with ground ring $R = \Z[E_1, E_2]$. This is the Frobenius algebra associated with $U(2)$-equivariant cohomology of $\mathbb{CP}^1$ \cite{Kh3}. It specializes to the original theory by setting $E_1 = E_2 = 0$ and to the Lee deformation \cite{Le} by setting $E_1= 0, E_2=-1$. An equivariant version of $\sl_3$-homology was constructed in \cite{MV}, and a generalization to $\sl_n$-homology can be found in \cite{Kr}.

In another direction, Asaeda-Przytycki-Sikora \cite{APS} defined homology for links in $I$-bundles over surfaces. The present paper concerns links in the solid torus, identified with $\A \times [0,1]$ where $\A = S^1 \times [0,1]$ is the annulus. The APS construction in this case is known as \emph{annular Khovanov homology} or \emph{annular APS  homology}. It is a triply graded theory; in addition to homological and quantum gradings, there is a third grading arising from the presence of non-contractible circles in $\A$. The APS annular chain complex may be obtained by applying to the cube of resolutions the \emph{annular TQFT}  
\[
\Gn \colon \BN(\A) \to \Z\ggmod,
\]
where $\BN(\A)$ is the Bar-Natan category of the annulus, and $\k \ggmod$ denotes the category of bigraded modules over a ring $\k$.

This paper extends annular Khovanov homology to the equivariant setting. We work with the Frobenius algebra 
\[
A_\al = \Z[\al_0, \al_1, X]/((X-\al_0)(X-\al_1))
\]
with ground ring $R_\al = \Z[\al_0, \al_1]$, which are the $U(1)\times U(1)$-equivariant cohomology of $\mathbb{CP}^1$ and of a point, respectively \cite{KR}. The Frobenius pair $(R_\al, A_\al)$ is an extension of $(R,A)$ by identifying $E_1, E_2$ with elementary symmetric polynomials in $\al_0, \al_1$,
\[
E_1\mapsto \al_0 + \al_1, E_2\mapsto \al_0\al_1,
\]
so that the polynomial $X^2 - E_1X+E_2 \in R[X]$ splits as $(X-\al_0)(X-\al_1)$ in $R_\al[X]$. We observe in Section \ref{sec:preliminary observation} that working over $(R,A)$ cannot produce an equivariant version of annular APS homology. There is a natural $U(1)\times U(1)$-equivariant analogue $\BN_\al(\A)$ of $\BN(\A)$ where the local relations are dictated by the structure of $A_\al$. 

Our main construction is a TQFT $\Ga$, which, when applied to the cube of resolutions of an annular link diagram, gives a $U(1) \times U(1)$-equivariant version of annular homology.
\begin{theorem}
There exists a functor $\Ga \colon \BN_\al(\A) \to R_\al\ggmod$ such that the following diagram commutes 
\begin{equation*}
\begin{tikzcd}
\BN_\al(\A) \ar[r, "\Ga"] \ar[d] & R_\al\ggmod \ar[d]\\
\BN(\A) \ar[r, "\Gn"] & \Z\ggmod
\end{tikzcd}
\end{equation*}
where the vertical arrows are obtained by setting $\al_0 = \al_1 = 0$. 
\end{theorem}\noindent

We define $\Ga$ by choosing a suitable basis and using a filtration induced by an additional annular grading, as in \cite{Ro}. Given a collection of disjoint simple closed curves $\Cs \subset \A$, each circle in $\Cs$ is assigned the module $A_\al$, with the module assigned to a trivial circle concentrated in annular degree zero. The essential circles in $\Cs$ are naturally ordered. For each essential circle $C$ in $\Cs$ we equip its module $A_\al$ with a distinguished homogeneous basis, either $\{1, X-\al_0\}$ or $\{1, X-\al_1\}$, depending in an alternating manner on the position of $C$. We show that maps assigned to cobordisms are non-decreasing with respect to the annular grading.

A feature of the equivariant theory is that the dotted product cobordism on a non-contractible circle in $\A$, 
\begin{equation*}
\begin{gathered}
\includegraphics{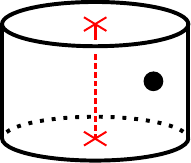}
\end{gathered}\, ,
\end{equation*}
is not sent to the zero map. Algebraically, this says that multiplication by $X$ on an essential circle is nonzero in the equivariant theory. On the other hand, this cobordism evaluates to zero in APS homology and also in the quantum annular homology \cite{BPW}.

The paper is organized as follows. In Section \ref{sec:Khovanov homology} we review Khovanov homology using the framework of the Bar-Natan cobordism category \cite{BN2}. The remaining parts of Section \ref{sec:some link homology theories} give an overview of Frobenius algebras $A$, $A_\al$, and $A_\ald$, following \cite{KR}. Section \ref{sec:annular Khovanov homology} reviews annular Khovanov homology. Our equivariant theory is defined in Section \ref{sec:defining the functor}. In Section \ref{sec:inverting D equiv annular} we study a further extension appearing in \cite{KR}, which is obtained by inverting an element $\D \in A_\al$. We prove an analogue of \cite[Theorem 4.2]{Le}, that for a $k$-component annular link, the homology obtained by inverting $\D$ is free of rank $2^k$. In Section \ref{sec:dotted TL} we recall the Temperley-Lieb category and its relation to annular Khovanov homology, following observations in \cite{BPW}. This perspective leads to a natural equivariant analogue of the Temperley-Lieb category and algebra, where strands may carry dots. 
\vskip1em
\textbf{Acknowledgements.} I would like to thank Mikhail Khovanov for suggesting this project, for many helpful discussions, and for comments on earlier versions of the paper. I would also like to thank my advisor Slava Krushkal. The author was supported by NSF RTG Grant DMS-1839968.

\section{Some link homology theories}\label{sec:some link homology theories}
We review Bar-Natan's approach to Khovanov homology and describe four Frobenius algebras, all of which have appeared in the literature and yield homology for links in $\R^3$. 

\subsection{Khovanov homology}\label{sec:Khovanov homology}

We start with a brief overview of the Bar-Natan category $\BN$ and the construction of the chain complex $[[D]]$ assigned to a link diagram $D$; for a complete treatment see \cite{BN2}. 

First, recall the (dotted) Bar-Natan category $\BN$. Let $I :=[0,1]$ denote the unit interval. Objects of $\BN$ are formal direct sums of formally graded collections of simple closed curves in the plane $\R^2$. Morphisms are matrices whose entries are formal $\Z$-linear combinations of dotted cobordisms properly embedded in $\R^2 \times I$, modulo isotopy relative to the boundary, and subject to the local relations shown in Figure \ref{fig:BN relations}. For the remainder of the paper, all cobordisms are assumed to possibly carry dots, unless specified otherwise.

Let $A_0= \Z[X]/(X^2)$. The trace
\[
\eps_0 \colon A_0 \to \Z,\, 1\mapsto 0,\, X\mapsto 1
\]
makes $A_0$ a Frobenius algebra, with comultiplication
\begin{align*}
A_0 & \to A_0 \o A_0\\
1 & \mapsto X \o 1 + 1\o X\\
X & \mapsto X\o X 
\end{align*}
This is the Frobenius algebra underlying $\sl_2$ link homology \cite{Kh1}.

\begin{figure}[H]
\centering
\begin{subfigure}[b]{.3\textwidth}
\begin{center}
\includegraphics{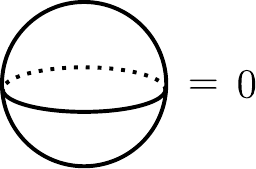}
\end{center}
\caption{Sphere}\label{fig:sphere}
\end{subfigure}%
\begin{subfigure}[b]{.4\textwidth}
\begin{center}
\includegraphics{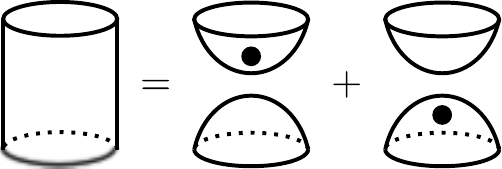}
\end{center}
\caption{Neck-cutting}\label{fig:neck-cutting}
\end{subfigure}\\
\vskip1em
\begin{subfigure}[b]{.3\textwidth}
\begin{center}
\includegraphics{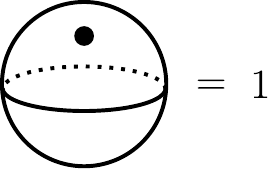}
\end{center}
\caption{Dotted sphere}\label{fig:dotted sphere}
\end{subfigure}%
\begin{subfigure}[b]{.35\textwidth}
\begin{center}
\includegraphics{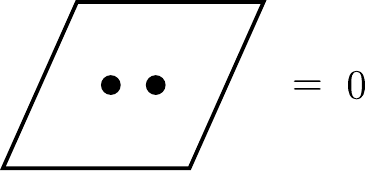}
\end{center}
\caption{Two dots}\label{fig:two dot}
\end{subfigure}
\caption{Relations in $\BN$}\label{fig:BN relations}
\end{figure}

The Bar-Natan relations (Figure \ref{fig:BN relations}) can be seen as arising from the structure of $A_0$ in the following way. Interpret the cup cobordism as the unit map
\[
\eta_0\colon \Z \to A_0,\, 1\mapsto 1,
\]
the cap as the trace $\eps_0$, and a dot as multiplication by $X\in A_0$. Then the sphere relation corresponds to 
\[
\eps_0(\eta_0(1))=0
\]
while the dotted sphere comes from $\eps_0(X) = 1$. The two dots relation corresponds to the relation $X^2 =0$ in $A_0$. Neck-cutting is a topological incarnation of the algebraic relation
\[
y = X \eps_0(y) + \eps_0(Xy),
\]
which holds for every $y\in A_0$. 

For a cobordism $S\subset \R^2 \times I$, let $d(S)$ denote the number of dots on $S$, and set the degree of $S$ to be 
\begin{equation}\label{eq:grading on cobs}
\deg(S) = -\chi(S) + 2d(S).
\end{equation}
Note that the relations in Figure \ref{fig:BN relations} are homogeneous. Define the quantum grading $\qdeg$ on $A_0$ by setting 
\begin{equation}\label{eq:1 and X gradings}
\qdeg(1) = -1 \hskip2em \qdeg(X) = 1.
\end{equation}

\begin{remark}\label{rmk:gradings A_0}
The grading elsewhere in the literature \cite{Kh1, BN1, BN2} is opposite that of the one appearing here. Moreover, viewing $A_0$ as an algebra, it is more natural to set $1$ and $X$ in degrees $0$ and $2$, respectively, to make the multiplication grading-preserving. However, when viewing $A_0$ as a module, degrees are balanced around $0$ as above. 
\end{remark}

For a ring $\k$, let $\k\gmod$ denote the category of $\Z$-graded $\k$-modules and graded maps (of any degree) between them. 
The Frobenius algebra $A_0$ defines a $(1+1)$-dimensional TQFT, and it descends to a graded, additive functor 
\begin{equation}\label{eq:tqft F}
\Fn \colon \BN \to \Z\gmod.
\end{equation}
which is $\Z$-linear on each morphism space. In fact, due to delooping \cite{BN3}, any such functor is determined by its value on the empty diagram.

Let $D$ be a diagram for an oriented link $L\subset \R^3$. We recall the construction $[[D]]$ from \cite{BN2}, which is a chain complex over the additive category $\BN$. One first forms the \emph{cube of resolutions} as follows. Label the crossings of $D$ by $1,\ldots, n$. Every crossing may be resolved in two ways, called the \emph{0-smoothing} and \emph{1-smoothing}, shown in Figure \ref{fig:0 and 1 smoothings}. For each $u = (u_1,\ldots, u_n)\in \{0,1\}^n$, perform the $u_i$-smoothing at the $i$-th crossing. The resulting diagram is a collection of disjoint simple closed curves in the plane $\R^2$, and we denote it by $D_u$. Thinking of elements of $\{0,1\}^n$ as vertices of an $n$-dimensional cube, decorate the vertex $u$ by the smoothing $D_u$. 

\begin{figure}
\centering
\includegraphics{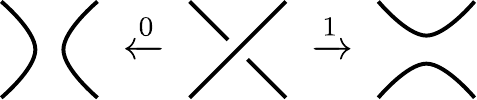}
\caption{The two smoothings at a crossing}\label{fig:0 and 1 smoothings}
\end{figure}

Next, let $u=(u_1,\ldots, u_n)$ and $v=(v_1,\ldots, v_n)$ be vertices which differ only in the $i$-th entry, where $u_i = 0$ and $v_i=1$. Then the diagrams $D_u$ and $D_v$ are the same outside of a small disk around the $i$-th crossing. There is a cobordism from $D_u$ to $D_v$, which is the obvious saddle ($1$-handle attachment) near the $i$-th crossing and the identity (product cobordism) elsewhere. Denote this cobordism by $d_{u,v}$, and decorate each edge of the $n$-dimensional cube by these saddle cobordisms. This forms a commutative cube in the category $\BN$. There is a way to assign $s_{u,v} \in \{0,1\}$ to each edge in the cube so that multiplying the edge map $d_{u,v}$ by $(-1)^{s_{u,v}}$ results in an anti-commutative cube (see \cite[Section 2.7]{BN2}).

For $u=(u_1,\ldots, u_n) \in \{0,1\}^n$, set $\lr{u} = \sum_i u_i$. Now, form the chain complex $[[D]]$ over $\BN$ by setting 
\[
[[D]]^i = \bigoplus_{\vert u \vert = i+ n_-} D_{u}\{n_- - n_+ -i \} 
\] 
in homological degree $i$, where $n_-$, $n_+$ denote the number of negative and positive crossings in $D$, and $\{-\}$ is the upwards grading shift in $\BN$. The differential is given on each summand by the edge map $(-1)^{s_{u,v}}d_{u,v}$. Anti-commutativity of the cube ensures that $[[D]]$ is a chain complex. 

The relations in Figure \ref{fig:BN relations} imply the $S$, $T$, and $4Tu$ relations from \cite{BN2}. 

\begin{theorem}\emph{(\cite[Theorem 1]{BN2})}\label{thm:[[D]] is invariant}
If diagrams $D$ and $D'$ are related by a Reidemeister move, then $[[D]]$ and $[[D']]$ are chain homotopy equivalent.
\end{theorem}

Thus to obtain link homology, it suffices to apply a functor from $\BN$ into an abelian category, and Theorem \ref{thm:[[D]] is invariant} guarantees that the homotopy class of the resulting chain complex is a link invariant. In particular, the TQFT \eqref{eq:tqft F} yields a chain complex  
\[
CKh(D) := \Fn([[D]])
\]
of graded abelian groups. After reversing the quantum grading, this is the chain complex appearing in \cite[Section 7]{Kh1}. 

\subsection{$U(2)$-equivariant Khovanov homology}\label{sec:equivariant Khovanov homology}

This section reviews the so-called $U(2)$-equivariant Frobenius pair, denoted $(R,A)$. Although this extension is of general importance, it is not necessary for our construction in Section \ref{sec:equivariant annular Khovanov homology}; in fact, in Section \ref{sec:preliminary observation} we note that an analogue of annular APS homology using $(R,A)$ is not possible. 

Consider the graded ring
\[
R = \Z[E_1, E_2] 
\]
with $\deg(E_1) = 2$, $\deg(E_2) = 4$. The $R$-algebra
\[
A = R[X]/(X^2 - E_1 X + E_2)
\]
equipped with the trace 
\[
\eps \colon A\to R,\, 1 \mapsto 0, X\mapsto 1
\]
is a Frobenius algebra over $R$. The rings $R$ and $A$ are the  $U(2)$-equivariant cohomology with $\Z$ coefficients of a point and $\mathbb{CP}^1$, respectively \cite{Kh3}. The Frobenius algebra $A$ determines a link homology theory as in Section \ref{sec:Khovanov homology}, obtained by applying the corresponding TQFT to the formal complex $[[D]]$.

\subsection{$U(1) \times U(1)$-equivariant  Khovanov homology}\label{sec:equivariant Khovanov homology extension}

In this section we review an extension of the Frobenius pair $(R,A)$. This extension was studied in \cite{KR} and is central to our construction in Section \ref{sec:equivariant annular Khovanov homology}. 

Let 
\[
R_\al = \Z[\al_0, \al_1],
\]
and consider the $R_\al$-algebra
\[
A_\al = R_\al[X]/((X-\al_0)(X-\al_1)).
\]
The trace 
\[
\eps_\al \colon A_\al \to R_\al,\, 1\mapsto 0,\, X\mapsto 1.
\]
makes $A_\al$ into a Frobenius algebra, with comultiplication
\begin{align*}
\Delta \colon  A_\al & \to A_\al\otimes A_\al\\
 1 & \mapsto (X-\al_0)\otimes 1 + 1\otimes (X-\al_1)\\
 X & \mapsto X\otimes X - \al_0 \al_1 1\otimes 1 .
\end{align*}
There is an inclusion $(R, A) \hookrightarrow (R_\al, A_\al)$ given by identifying $E_1, E_2\in R$ with the elementary symmetric polynomials in $R_\al$,
\[
E_1 \mapsto \al_0 + \al_1 \hskip2em E_2 \mapsto \al_0\al_1.
\]
As noted in \cite{KR}, $R_\al$ and $A_\al$ are the $U(1) \times U(1)$-equivariant cohomology with $\Z$ coefficients of a point and $2$-sphere $S^2$, respectively.

Let $\BN_\al$ denote the Bar-Natan category subject to relations coming from $A_\al$. Objects of $\BN_\al$ are formal direct sums of formally graded collections of simple closed curves in the plane $\R^2$. Morphisms are matrices whose entries are formal $R_\al$-linear combinations of dotted cobordisms properly embedded in $\R^2 \times I$, modulo isotopy relative to the boundary, and subject to the local relations shown in Figure \ref{fig:BNa relations}. As outlined in Section \ref{sec:Khovanov homology}, these topological relations correspond to algebraic relations in the Frobenius algebra $A_\al$. 

\begin{remark}\label{rmk:cob cat BNa is induced from U(2)}
We note that $\BN_\al$ is induced from the corresponding $U(2)$-equivariant cobordism category with relations dictated by $(R,A)$, since the relations involve only symmetric polynomials in $\al_0, \al_1$. 
\end{remark}

\begin{figure}
\centering
\begin{subfigure}[b]{.3\textwidth}
\begin{center}
\includegraphics{BN_rel1.pdf}
\end{center}
\caption{Sphere}\label{fig:alpha sphere}
\end{subfigure}%
\begin{subfigure}[b]{.7\textwidth}
\begin{center}
\includegraphics{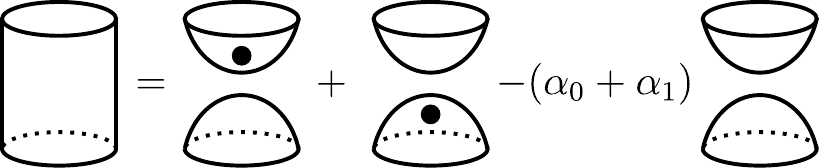}
\end{center}
\caption{Neck-cutting}\label{fig:alpha neck-cutting}
\end{subfigure}\\
\vskip1em
\begin{subfigure}[b]{.3\textwidth}
\begin{center}
\includegraphics{BN_rel3.pdf}
\end{center}
\caption{Dotted sphere}\label{fig:alpha dotted sphere}
\end{subfigure}%
\begin{subfigure}[b]{.65\textwidth}
\begin{center}
\includegraphics{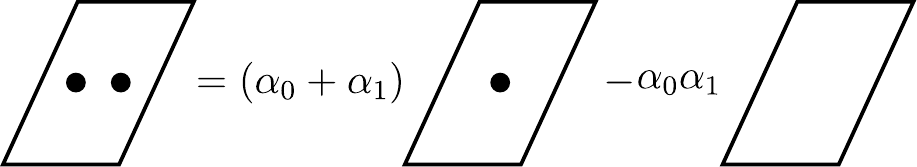}
\end{center}
\caption{Two dots}\label{fig:alpha two dot}
\end{subfigure}
\caption{Relations in $\BN_\al$}\label{fig:BNa relations}
\end{figure}

For a cobordism $S\subset \R^2\times I$, define the degree of $S$ as in \eqref{eq:grading on cobs}, and put $\al_0, \al_1\in R_\al$ in degree $2$. Note that the relations in Figure \ref{fig:BNa relations} are homogeneous. The algebra $A_\al$ is a free $R_\al$-module with basis $\{1, X\}$. Using the same notation as in \eqref{eq:1 and X gradings}, define a grading $\qdeg$ on $A_\al$ by setting 
\begin{equation}\label{eq:gradings A alpha}
\qdeg(1) = -1 \hskip2em \qdeg(X) = 1.
\end{equation}

\begin{remark}\label{rmk:gradings}
Viewing $A_\al$ as an $R_\al$-algebra, it is more natural to set $1$ and $X$ in degrees $0$ and $2$, respectively, so that multiplication in $A_\al$ is grading-preserving. When viewing $A_\al$ as an $R_\al$-module with homogeneous basis $\{1, X\}$ according to the grading \eqref{eq:gradings A alpha}, the elements $X-\al_0$ and $X-\al_1$ should be interpreted as $X - \al_0 \cdot 1$ and $X-\al_1 \cdot 1$, which are homogeneous of degree $1$. In either case, multiplication by $X$ is a degree $2$ endomorphism of $A_\al$.  
\end{remark}

The Frobenius algebra $A_\al$ defines a two-dimensional TQFT, and it descends to a graded, additive functor 
\begin{equation}\label{eq:tqft Falpha}
\Fa \colon \BN_\al \to R_\al\gmod
\end{equation}
which is $R_\al$-linear on each morphism space. Moreover, the following diagram commutes
\begin{equation}\label{eq:diagram1}
\begin{tikzcd}
\BN_\al \ar[r,"\Fa"] \ar[d] & R_\al\gmod \ar[d]\\
\BN \ar[r, "\Fn"] \ar[r] & \Z\gmod
\end{tikzcd}
\end{equation}
where the vertical maps are obtained by setting $\al_0 = \al_1=0$.

Given a diagram $D$ for an oriented link $L \subset \R^3$, form the chain complex $[[D]]$ as described in Section \ref{sec:Khovanov homology}. We may view $[[D]]$ as a chain complex over $\BN_\al$. The relations in Figure \ref{fig:BNa relations} imply the $S$, $T$, and $4Tu$ relations from \cite{BN2}, so by \cite[Theorem 1]{BN2}, the homotopy class of $[[D]]$ is an invariant of $L$. It follows that the chain complex obtained by applying $\Fa$ to $[[D]]$ is an invariant of $L$ up to chain homotopy equivalence.

\subsection{Inverting the discriminant and Lee homology}\label{sec:inverting D} 

We recall from \cite{KR} a further extension of the Frobenius pair $(R_\al, A_\al)$. Let 
\begin{equation}\label{eq:D}
\D = (\al_0 -\al_1)^2
\end{equation}
denote the discriminant of the quadratic polynomial $(X-\al_0)(X-\al_1) \in R_\al[X]$. Let 
\[
R_\ald = R_\al [ \D^{-1}]
\]
denote the ring obtained by inverting $\D$ (equivalently, one may invert $\al_0 - \al_1$) and let
\[
A_\ald = A_\al \o_{R_\al} R_\ald
\]
be the extension of $A_\al$ to an $R_\ald$-algebra. Let $\Fad$ denote the composition 
\[
\BN_\al \rar{\Fa} R_\al \gmod \to R_\ald \gmod
\]
where the second functor is extension of scalars, $(-)\o_{R_\al} R_\ald$. For a link $L\subset \R^3$ with diagram $D$, let 
\[
CKh_{\ald}(D) := \Fad([[D]])
\]
denote the resulting chain complex. It is an invariant of $L$ up to chain homotopy equivalence, and we will denote its homology by $Kh_{\ald}(L)$. 

The elements 
\begin{equation}\label{eq:e_0 and e_1}
e_0  = \frac{X-\al_0}{\al_1-\al_0},\hskip1em e_1 = \frac{X - \al_1}{\al_0 - \al_1} \in A_\ald. 
\end{equation}
form a basis for $A_\ald$ and satisfy 
\[
e_0 + e_1 = 1,\ e_0^2 = e_0,\ e_1^2 = e_1,\ e_0e_1=0,
\]
so that the algebra $A_\ald$ decomposes as a product, $A_\ald = R_\ald e_0 \times R_\ald e_1$. With respect to the basis $\{e_0, e_1\}$, comultiplication in $A_\ald$ is simply given by
\begin{equation}\label{eq:comultiplication D}
\begin{split}
\Delta(e_0) &= (\al_1 - \al_0) e_0 \o e_0 \\
\Delta(e_1) &= (\al_0 - \al_1) e_1\o e_1.
\end{split}
\end{equation}

As noted in \cite[Section 1.2]{KR}, the TQFT $\Fad$ is essentially the Lee deformation \cite{Le}. By \cite[Theorem 4.2]{Le}, the Lee homology of a $k$-component link is free (over $\Q$) of rank $2^k$. A quick alternate proof can be found in the final remark in \cite{We}. The following is stated in \cite{KR} without proof, but the arguments in \cite{We} apply without modification. 

\begin{proposition}\label{prop:Lee rank}
For a link $L\subset \R^3$ with $k$ components, the homology $Kh_{\ald}(L)$ is a free $R_\ald$-module of rank $2^k$.
\end{proposition}

\section{Annular Khovanov homology}\label{sec:annular Khovanov homology}
We give an overview of annular Khovanov homology, also known as annular Asaeda-Przytycki-Sikora (APS) homology. It was originally defined in \cite{APS} as part of a broader categorification of the Kauffman bracket skein module of $I$-bundles over surfaces. A convenient reference for the annular setting is \cite{GLW}. 

Let $\A = S^1\times I$ denote the annulus. An \textit{annular link} is a link in the thickened annulus $\A\times I$, and its diagram is a projection onto the first factor of $\A \times I$. Embed $\A$ standardly in $\R^2$ as 
\[
\A = \{ x\in \R^2 \mid 1\leq \lr{x} \leq 2\},
\]
so that an annular link diagram and all of its smoothings are drawn in the punctured plane $\R^2\setminus (0,0)$. We represent the annulus in the plane by simply indicating the puncture using the symbol $\times$. Figure \ref{fig:annular link} illustrates an example of an annular link diagram. By a \emph{circle} in $\A$ we mean a smoothly and properly embedded $S^1$ in $\A$. There are two kinds of circles in $\A$: \emph{trivial} circles, which are contractible in $\A$, and \emph{essential} ones, which are not contractible.

\begin{figure}
\centering
\includegraphics{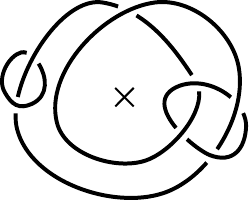}
\caption{An annular link diagram}\label{fig:annular link}
\end{figure}

Let $\BN(\A)$ denote the Bar-Natan category of the annulus. Its objects are formal direct sums of formally bigraded collections of simple closed curves in $\A$. Morphisms are matrices whose entries are formal $\Z$-linear combinations of dotted cobordisms properly embedded in $\A \times I$, modulo isotopy relative to the boundary, and subject to the local relations shown in Figure \ref{fig:BN relations}. The bidegree of a cobordism $S\subset \A \times I$ is defined to be
\begin{equation}\label{eq:annular bidegree}
(-\chi(S) + d(S), 0),
\end{equation}
where $d(S)$ is the number of dots on $S$. 

For a ring $\k$, denote by $\k\ggmod$ the category of $\Z \times \Z$-graded $\k$-modules and graded maps (of any bidegree) between them. We now describe the annular TQFT
\[
\Gn \colon \BN(\A) \to \Z\ggmod,
\]
which will be additive, graded, and $\Z$-linear on each morphism space. 

Let $\Cs \subset \A$ be a collection of $n$ trivial and $m$ essential circles. Embed $\A \times I$ standardly into $\R^2 \times I$, and apply the TQFT $\Fn$ from Section \ref{sec:Khovanov homology},
\[
\Fn(\Cs) = A_0^{\o n} \o A_0^{\o m}. 
\]
Define a second grading, called the \emph{annular} grading and denoted $\adeg$, on $\Fn(\Cs)$ in the following way. A tensor factor $A_0$ corresponding to a trivial circle is concentrated in annular degree $0$. For a factor $A_0$ corresponding to an essential circle, let
\[
v_0 = 1,\hskip2em v_1 = X.
\]
denote a basis for this copy of $A_0$, and set 
\[
\adeg(v_0) = -1 \hskip2em \adeg(v_1) = 1.
\]
Bigradings are summarized in Figure \ref{fig:bigradings}. 

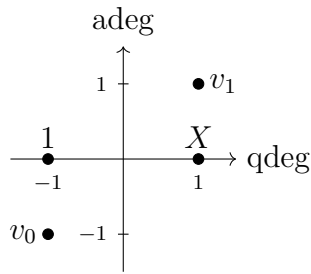
\begin{figure}
\centering
\begin{tikzpicture}
\draw[->] (-1.5,0) -- (1.5,0) node[right] {$\qdeg$};
\foreach \x in {-1, 1}
\draw[shift={(\x,0)}] (0pt,2pt) -- (0pt,-2pt) node[below] {\tiny $\x$};
\draw[->] (0,-1.5) -- (0,1.5) node[above] {$\adeg$};
\foreach \y in {-1,1}
\draw[shift={(0,\y)}] (2pt,0pt) -- (-2pt,0pt) node[left] {\tiny $\y$};
\filldraw[black] (1,1) circle (2pt) node[anchor=west] {$v_1$};
\filldraw[black] (-1,-1) circle (2pt) node[anchor=east] {$v_0$};
\filldraw[black] (1,0) circle (2pt) node[anchor=south] {$X$};
\filldraw[black] (-1,0) circle (2pt) node[anchor=south] {$1$};
\end{tikzpicture}
\caption{Bigradings, where $\{1, X\}$ corresponds to trivial circles, and $\{v_0, v_1\}$ correspond to essential ones. See also Remark \ref{rmk:gradings A_0}.}\label{fig:bigradings}
\end{figure}

The underlying abelian group of $\Gn(\Cs)$ is $\Fn(\Cs)$, and the bigrading is given by $(\qdeg, \adeg)$. For a cobordism $S\subset \A\times I$, first view $S$ as a surface in $\R^2 \times I$ and consider the map $\Fn(S)$. It is shown in \cite[Section 2]{Ro} that $\Fn(S)$ splits as a sum
\begin{equation}\label{eq:non alpha cob split}
\Fn(S) = \Fn(S)_0 + \Fn(S)_+
\end{equation}
where $\Fn(S)_0$ preserves $\adeg$ and $\Fn(S)_+$ increases $\adeg$. Set 
\[
\Gn(S) := \Fn(S)_0
\]
to be the $\adeg$-preserving part. It follows from \eqref{eq:non alpha cob split} that $\Gn$ is functorial with respect to composition of cobordisms. By construction, $\Gn(S)$ is a map of bidegree 
\[
(-\chi(S)+2d(S), 0)
\]
so the functor $\Gn$ is degree preserving on morphism spaces. We will refer to $\Gn$ as the \emph{annular TQFT}. 

To distinguish the bigraded modules assigned to trivial and essential circles, write 
\[
V = \Gn(C)
\]
if $C$ is an essential circle, with basis written as $\{v_0, v_1\}$, and keep the notation $A_0 = \Gn(C)$ when $C$ is trivial. Then if $\Cs \subset \A$ consists of $n$ trivial and $m$ essential circles, the module assigned to $\Cs$ is
\[
\Gn(\Cs) = A_0^{\o n} \o V^{\o m}.
\]

Given a diagram $D$ for an oriented annular link $L$, form the chain complex $[[D]]$ as described in Section \ref{sec:Khovanov homology}. The construction is completely local and crossings are away from the puncture $\times$. Thus we may view $[[D]]$ as a chain complex over $\BN(\A)$, with $\Z$-grading shifts $\{-\}$ in $\BN$ rewritten as a $\Z\times \Z$-grading shifts $\{-,0\}$ in $\BN(\A)$. Isotopies of annular links are described by Reidemeister moves away from the puncture, and it follows that the homotopy class of $[[D]]$, viewed as a chain complex over $\BN(\A)$, is an invariant of $L$. Therefore the chain complex 
\begin{equation}\label{eq:annular chain complex}
CKh^\A(D) := \Gn([[D]])
\end{equation}
is an invariant of $L$ up to chain homotopy equivalence.

An \emph{elementary cobordism} is one that has a single non-degenerate critical point with respect to the height function $\A \times I \to I$. It consists of a union of a product cobordism and a single cup, cap, or saddle. An elementary cobordism $S$ with $\d S$ consisting of trivial circles in $\A$ is assigned the same map by $\Fn$ and $\Gn$. We record the maps assigned to the four elementary saddles involving at least one essential circle, Figure \ref{fig:el saddles}. The vertical red arc is the central axis of $\A \times I \subset \R^2\times I$.

\noindent\begin{minipage}{.5\linewidth}
\begin{equation}\label{eq:formula1}
  \begin{split}
   V\o A_0 & \rar{\hyperref[fig:type1]{\operatorname{(I)}}} V \\  v_0 \o 1 & \mapsto v_0 \\ v_1\o 1 &\mapsto v_1 \\
    v_0 \o X & \mapsto 0 \\
    v_1 \o X & \mapsto 0
  \end{split}
\end{equation}
\end{minipage}%
\begin{minipage}{.35\linewidth}
\begin{equation}\label{eq:formula2}
  \begin{split}
    V \o V & \rar{\hyperref[fig:type2]{\operatorname{(II)}}} A_0 \\  v_0 \o v_0 &\mapsto 0 \\
v_1 \o v_0 & \mapsto X \\
v_0 \o v_1 & \mapsto X \\
v_1 \o v_1 & \mapsto 0
  \end{split}
\end{equation}
\end{minipage}

\vskip1em
 
\noindent\begin{minipage}{.49\linewidth}
\begin{equation}\label{eq:formula3}
  \begin{split}
  V & \rar{\hyperref[fig:type3]{\operatorname{(III)}}} V \o A_0 \\
v_0 & \mapsto v_0 \o X \\
v_1 & \mapsto v_1 \o X
  \end{split}
\end{equation}
\end{minipage} 
\begin{minipage}{.4\linewidth}
\begin{equation}\label{eq:formula4}
  \begin{split}
    A_0 & \rar{\hyperref[fig:type4]{\operatorname{(IV)}}} V\o V \\
1 &\mapsto v_0 \o v_1 + v_1 \o v_0\\
X & \mapsto 0
  \end{split}
\end{equation}
\end{minipage}
\vskip1em

\begin{figure}
\begin{subfigure}[b]{.2\textwidth}
\begin{center}
\includegraphics{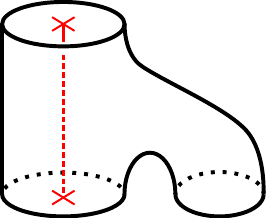}
\end{center}
\caption{Type (I)}\label{fig:type1}
\end{subfigure}
\begin{subfigure}[b]{.2\textwidth}
\begin{center}
\includegraphics{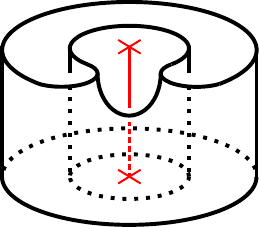}
\end{center}
\caption{Type (II)}\label{fig:type2}
\end{subfigure}
\begin{subfigure}[b]{.2\textwidth}
\begin{center}
\includegraphics{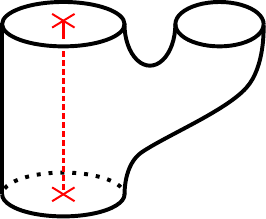}
\end{center}
\caption{Type (III)}\label{fig:type3}
\end{subfigure}
\begin{subfigure}[b]{.2\textwidth}
\begin{center}
\includegraphics{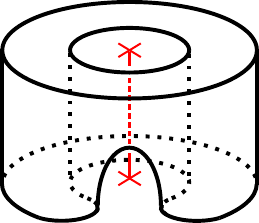}
\end{center}
\caption{Type (IV)}\label{fig:type4}
\end{subfigure}
\caption{Saddles involving essential circles}\label{fig:el saddles}
\end{figure}

From \eqref{eq:formula1}, we see that $X$ acts trivially on any essential circle. It follows that a cobordism with a component that carries a dot and a closed curve which is nonzero in $\pi_1(\A \times I)$ is assigned the zero map. Thus $\Gn$ factors through the relation shown in Figure \ref{fig:Boerner}, called Boerner's relation \cite{Bo}. Indeed, for an essential circle $C\subset \A$, there are no nonzero endomorphisms of $\Gn(C)$ with bidegree $(2,0)$. 

\begin{figure}
\centering
\includegraphics{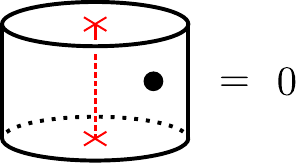}
\caption{Boerner's relation}\label{fig:Boerner}
\end{figure}

The category $\BN(\A)$ is monoidal, with monoidal product  given by taking two copies $\A_1, \A_2$ of $\A$ and gluing the boundary component $S^1\times \{1\}$ of $\A_1$ to the boundary component $S^1 \times \{0\}$ of $\A_2$. The annular TQFT $\Gn$ is evidently monoidal.

\section{Equivariant annular Khovanov homology}\label{sec:equivariant annular Khovanov homology}

We are interested in an annular version of the theory outlined in Section \ref{sec:equivariant Khovanov homology extension}. Precisely, the goal is to fill in the dashed arrow in the diagram
\[
\begin{tikzcd}\label{eq:dashed diagram}
\BN_\al(\A) \ar[r, dashed, "\Ga"] \ar[d] & R_\al\ggmod \ar[d]\\
\BN(\A) \ar[r, "\Gn"] & \Z\ggmod
\end{tikzcd}
\]
where the vertical arrows are obtained by setting $\al_0 = \al_1 = 0$. Section \ref{sec:preliminary observation} justifies working with the extension $(R_\al, A_\al)$ rather than $(R,A)$. The desired functor $\Ga$ is defined in Section \ref{sec:defining the functor}. Maps assigned to saddle cobordisms can be found in \eqref{eq:equiv formula1}--\eqref{eq:equiv formula4}. In Section \ref{sec:inverting D equiv annular} we invert $\D$ in the annular theory and show that the rank of the resulting homology depends only on the number of components.

\subsection{A preliminary observation}\label{sec:preliminary observation}

Before defining our equivariant annular TQFT, we note that the $U(2)$-equivariant Frobenius pair $(R,A)$ from Section \ref{sec:equivariant Khovanov homology} does not admit such a lift, under the minor assumption that modules assigned to circles are free. 

The ring $R=\Z[E_1, E_2]$ can be made bigraded, with bidegrees of $E_1$ and $E_2$ given by $(2,0)$ and $(4,0$), respectively. Let $M$ be a free $\Z \times \Z$-graded $R$-module with basis $m_-, m_+$ in bidegrees $(-1,-1)$ and $(1,1)$, respectively. Suppose $g \colon M\to M$ is an $R$-linear map of bidegree $(2,0)$. Then necessarily 
\begin{equation}\label{eq:map must be}
g(m_-) = n E_1 m_-
\end{equation}
for some $n\in \Z$. In particular, if $M$ is the module assigned to a single essential circle and $g$ is the map assigned to the cobordism in Figure \ref{fig:dotted id}, then the relation $X^2 - E_1X + E_2 = 0$
in $A$ implies 
\begin{equation}\label{eq:two dots must be}
g^2(m_-) - E_1 g(m_-) + E_2 m_- = 0.
\end{equation}
However, \eqref{eq:map must be} and \eqref{eq:two dots must be} are incompatible. 

\begin{figure}
\centering
\includegraphics{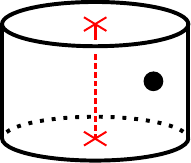}
\caption{Dotted product cobordism on an essential circle}\label{fig:dotted id}
\end{figure}

\subsection{The equivariant annular TQFT $\mathcal{G}_\alpha$}\label{sec:defining the functor}

Let $\BN_\al(\A)$ denote the Bar-Natan category of the annulus subject to the relations determined  by $A_\al$. Its objects are formal direct sums of formally bigraded collections of simple closed curves in $\A$. Morphisms are matrices whose entries are formal $R_\al$-linear combinations of dotted cobordisms properly embedded in $\A \times I$, modulo isotopy relative to the boundary, and subject to the local relations shown in Figure \ref{fig:BNa relations}. The bidegree of a cobordism $S\subset \A \times I$ is given by \eqref{eq:annular bidegree}. For an oriented annular link $L$ with diagram $D$, the formal complex $[[D]]$ over $\BN_\al(\A)$ is an invariant of $L$ up to chain homotopy equivalence.

Let $\Cs \subset \A$ be a collection of circles, and view $\Cs$ as embedded in $\R^2$. Consider $\Fa(\Cs)$ with the following additional annular grading, denoted $\adeg$ as in Section \ref{sec:equivariant Khovanov homology extension}. Define elements of $A_\al$, 
\begin{align*}
\v_0 = 1, \hskip2em \v_1 = X - \al_0, \\
\v_0'= 1, \hskip2em \v_1'= X - \al_1,
\end{align*}
with the annular gradings 
\begin{equation}
\adeg(\v_0) = \adeg(\v_0') = -1, \hskip2em \adeg(\v_1) = \adeg(\v_1') = 1. 
\end{equation}

\begin{remark}
The notation $v_0, v_1$ was also used in Section \ref{sec:annular Khovanov homology}. Setting $\al_0 = \al_1 = 0$ in the above expressions recovers $v_0, v_1$ in the non-equivariant setting. 
\end{remark}

Both $\{\v_0, \v_1\} = \{1, X - \al_0\}$ and $\{\v_0', \v_1'\} = \{1, X - \al_1\}$ is an $R_\al$-basis for $A_\al$. Together with the quantum grading, these equip $A_\al$ with two (isomorphic) structures of a bigraded $R_\al$-module, with the bigrading given by $(\qdeg, \adeg)$. The ground ring $R_\al$ lies in annular degree $0$. 

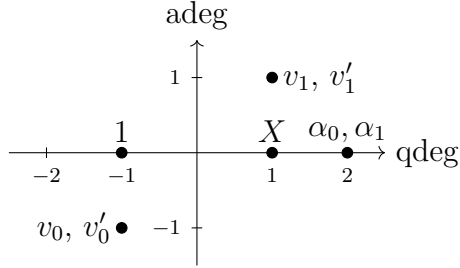
\begin{figure}
\centering
\begin{tikzpicture}
\draw[->] (-2.5,0) -- (2.5,0) node[right] {$\qdeg$};
\foreach \x in {-2, -1, 1, 2}
\draw[shift={(\x,0)}] (0pt,2pt) -- (0pt,-2pt) node[below] {\tiny $\x$};
\draw[->] (0,-1.5) -- (0,1.5) node[above] {$\adeg$};
\foreach \y in {-1,1}
\draw[shift={(0,\y)}] (2pt,0pt) -- (-2pt,0pt) node[left] {\tiny $\y$};
\filldraw[black] (1,1) circle (2pt) node[anchor=west] {$\v_1$, $\v_1'$};
\filldraw[black] (-1,-1) circle (2pt) node[anchor=east] {$\v_0$, $\v_0'$};
\filldraw[black] (1,0) circle (2pt) node[anchor=south] {$X$};
\filldraw[black] (-1,0) circle (2pt) node[anchor=south] {$1$};
\filldraw[black] (2,0) circle (2pt) node[anchor=south] {$\al_0, \al_1$};
\end{tikzpicture}
\caption{Bigradings, where $\{1, X\}$ corresponds to trivial circles, and $\{v_0, v_1\}$, $\{v_0', v_1'\}$ correspond to essential ones. See also Remark \ref{rmk:gradings}.}\label{fig:bigradings alpha}
\end{figure}

Let $\Cs \subset \A$ consist of $n$ trivial and $m$ essential circles, with the essential circles ordered from innermost (closest to the puncture $\times$) to outermost. Define the annular grading on 
\[
\Fa(\Cs) = A_\al^{\o n} \o A_\al^{\o m}
\]
by declaring that every copy of $A_\al$ corresponding to a trivial circle is concentrated in annular degree $0$ and that the copy of $A_\al$ corresponding to the $i$-th essential circle $( 1\leq i \leq m)$ is given the homogeneous basis 
\[
\{\v_0, \v_1\} = \{1, X - \al_0\}
\]
if $i$ is odd and 
\[
\{\v_0', \v_1'\} = \{1, X - \al_1\}
\]
if $i$ is even. In other words, the essential circles are assigned the homogeneous bases $\{ 1, X- \al_0\}$ or $\{1, X - \al_1\}$ in an alternating manner, with the innermost circle assigned $\{1, X- \al_0\}$. Bigradings are summarized in Figure \ref{fig:bigradings alpha}

As in Section \ref{sec:annular Khovanov homology}, it is convenient to distinguish the modules assigned to essential and trivial circles. Let $V_\al$ and $V'_\al$ denote the module $A_\al$ with homogeneous bases $\{\v_0, \v_1\}$ and $\{\v_0', \v_1'\}$, respectively. Then for a collection of circles $\Cs \subset \A$, the $i$-th essential circle in $\Cs$ is assigned $V_\al$ if $i$ is odd and $V_\al'$ if $i$ is even. We reserve the notation $A_\al$ for the module assigned to a trivial circle. Note that interchanging $\al_0 \lrar \al_1$ also interchanges $\v_0 \lrar \v_0'$ and $\v_1 \lrar \v_1'$.

\begin{lemma}\label{lem:el cobs split}

Let $S\subset \A\times I$ be an elementary cobordism. Viewing $S$ as a cobordism in $\R^2 \times I$, the map $\Fa(S)$ splits as a sum
\[
\Fa(S) = \Fa(S)_0 + \Fa(S)_2
\]
where $\Fa(S)_0$ preserves $\adeg$ and $\Fa(S)_2$ increases $\adeg$ by $2$.
\end{lemma}

\begin{proof}
 
If the saddle component of $S$ involves only trivial circles then the claim is immediate, since $\Fa(S) = \Fa(S)_0$ in this case. We verify the claim for the four elementary cobordisms in Figure \ref{fig:el saddles} by rewriting $\Fa(S)$ in terms of the bases for the circles involved. Terms where $\adeg$ is increased by $2$ are boxed.

\noindent\begin{minipage}{.5\textwidth}
\begin{equation*}
  \begin{split}
   V_\al \o A_\al & \rar{\hyperref[fig:type1]{\operatorname{(I)}}} V_\al \\  \v_0 \o 1 & \mapsto \v_0 \\ \v_1\o 1 &\mapsto \v_1 \\
    \v_0 \o X & \mapsto \al_0 \v_0 +\fbox{$\v_1$} \\
    \v_1 \o X & \mapsto \al_1 \v_1
  \end{split}
\end{equation*}
\end{minipage}%
\begin{minipage}{.38\textwidth}
\begin{equation*}
  \begin{split}
   V_\al \o V_\al' & \rar{\hyperref[fig:type2]{\operatorname{(II)}}} A_\al \\  \v_0 \o \v_0' &\mapsto \fbox{$1$} \\
\v_1 \o \v_0' & \mapsto X - \al_0 \\
\v_0 \o \v_1' & \mapsto X - \al_1  \\
\v_1 \o \v_1' & \mapsto 0
  \end{split}
\end{equation*}
\end{minipage}

\vskip1em

\noindent\begin{minipage}{.49\linewidth}
\begin{equation*}
  \begin{split}
   V_\al & \rar{\hyperref[fig:type3]{\operatorname{(III)}}} V_\al \o A_\al \\
\v_0 & \mapsto \v_0 \o X - \al_1 \v_0 \o 1 + \fbox{$\v_1 \o 1$}\\
\v_1 & \mapsto \v_1 \o X - \al_0 \v_1 \o 1
  \end{split}
\end{equation*}
\end{minipage} 
\begin{minipage}{.5\linewidth}
\begin{equation*}
  \begin{split}
   A_\al & \rar{\hyperref[fig:type4]{\operatorname{(IV)}}} V_\al\o V_\al' \\
1 &\mapsto \v_0 \o \v_1' + \v_1 \o \v_0'\\
X & \mapsto  \al_0 \v_0 \o \v_1' + \al_1 \v_1 \o \v_0' + \fbox{$\v_1 \o v_1'$}
  \end{split}
\end{equation*}
\end{minipage}
\vskip1em
\noindent

Our assignment for essential circles depends on nesting, so strictly speaking the above calculations do not handle all cases. However, note that for types \hyperref[fig:type3]{(I)} and \hyperref[fig:type3]{(II)}, the position of the essential circle does not change, and for types \hyperref[fig:type3]{(III)} and \hyperref[fig:type3]{(IV)}, the two essential circles involved in the saddle must be consecutive in the ordering. Thus a full verification amounts to interchanging $\v_0 \lrar \v_0'$, $\v_1 \lrar \v_1'$ in the input of above maps. One may check that this amounts to interchanging $\v_0 \lrar \v_0'$, $\v_1 \lrar \v_1'$, and $\al_0 \lrar \al_1$ in the output.
\end{proof}

\begin{corollary}\label{cor:cobs split}  
\begin{enumerate}[label= \emph{(\arabic*)}]
\item Let $S\subset \A \times I$ be a cobordism. Viewing $S$ as a cobordism in $\R^2 \times I$, the map $\Fa(S)$ splits as a sum
\[
\Fa(S) = \Fa(S)_0 + \Fa(S)_+
\]
where $\Fa(S)_0$ preserves $\adeg$ and $\Fa(S)_+$ increases $\adeg$. 
\item Let $S_1, S_2 \subset \A \times I$ be composable cobordisms. Then
\[
\Fa(S_2 S_1)_0 = \Fa(S_2)_0 \Fa(S_1)_0.
\] 
\end{enumerate}
\end{corollary}

\begin{proof}
For $(1)$, write $S$ as a composition $S = S_n \cdots S_1$ where each $S_i$ is an elementary cobordism. Functoriality of $\Fa$ and Lemma \ref{lem:el cobs split} yield

\begin{align*}
\Fa(S) &= \Fa(S_n)  \cdots  \Fa(S_1)\\
& = \left( \Fa(S_n)_0 + \Fa(S_n)_2\right)  \cdots  \left( \Fa(S_1)_0 + \Fa(S_1)_2\right)\\
& = \Fa(S_n)_0 \cdots  \Fa(S_1)_0 + \text{terms that increase } \adeg .
\end{align*}
Therefore 
\[
\Fa(S)_0 = \Fa(S_n)_0 \cdots  \Fa(S_1)_0
\]
is the desired $\adeg$-preserving part, and the remaining terms constitute $\Fa(S)_+$. Statement (2) follows from (1) in a similar fashion.
\end{proof}

We are now ready for the main theorem. 

\begin{theorem}\label{thm:equivariant annular}
There exists a functor $\Ga \colon \BN_\al(\A) \to R_\al\ggmod$ such that the following diagram commutes 
\begin{equation*}
\begin{tikzcd}
\BN_\al(\A) \ar[r, "\Ga"] \ar[d] & R_\al\ggmod \ar[d]\\
\BN(\A) \ar[r, "\Gn"] & \Z\ggmod
\end{tikzcd}
\end{equation*}
where the vertical arrows are obtained by setting $\al_0 = \al_1 = 0$. 
\end{theorem}

\begin{proof}
For a collection of circles $\Cs\subset \A$, set 
\[
\Ga(S) := \Fa(\Cs),
\]
with the bigrading $(\qdeg, \adeg)$ as defined earlier in this section. For a cobordism $S\subset \A \times I$, set 
\[
\Ga(S) := \Fa(S)_0
\]
as in Corollary \ref{cor:cobs split} (1). That $\Ga$ is well-defined on cobordisms and factors through the relations in $\BN_\al(\A)$ follows from the analogous statements for $\Fa$. Corollary \ref{cor:cobs split} (2) implies functoriality of $\Ga$. Finally, commutativity of the diagram follows from deleting the boxed terms and setting $\al_0 = \al_1 =0$ in the maps appearing in the proof of Lemma \ref{lem:el cobs split}, and comparing the result with the maps \eqref{eq:formula1}--\eqref{eq:formula4}. 
\end{proof}

Maps assigned to the four elementary saddles in Figure \ref{fig:el saddles} are recorded below. The full set of maps -- that is, if other essential circles are present -- can be obtained by interchanging $\al_0 \lrar \al_1$. 

\noindent\begin{minipage}{.5\textwidth}
\begin{equation}\label{eq:equiv formula1}
  \begin{split}
   V_\al \o A_\al & \rar{\hyperref[fig:type1]{\operatorname{(I)}}} V_\al \\  \v_0 \o 1 & \mapsto \v_0 \\ \v_1\o 1 &\mapsto \v_1 \\
    \v_0 \o X & \mapsto \al_0 \v_0 \\
    \v_1 \o X & \mapsto \al_1 \v_1
  \end{split}
\end{equation}
\end{minipage}%
\begin{minipage}{.38\textwidth}
\begin{equation}\label{eq:equiv formula2}
  \begin{split}
   V_\al \o V_\al' & \rar{\hyperref[fig:type2]{\operatorname{(II)}}} A_\al \\  \v_0 \o \v_0' &\mapsto 0 \\
\v_1 \o \v_0' & \mapsto X - \al_0 \\
\v_0 \o \v_1' & \mapsto X - \al_1  \\
\v_1 \o \v_1' & \mapsto 0
  \end{split}
\end{equation}
\end{minipage}

\vskip1em

\noindent\begin{minipage}{.49\linewidth}
\begin{equation}\label{eq:equiv formula3}
  \begin{split}
   V_\al & \rar{\hyperref[fig:type3]{\operatorname{(III)}}} V_\al \o A_\al \\
\v_0 & \mapsto \v_0 \o X - \al_1 \v_0 \o 1 \\
\v_1 & \mapsto \v_1 \o X - \al_0 \v_1 \o 1
  \end{split}
\end{equation}
\end{minipage} 
\begin{minipage}{.5\linewidth}
\begin{equation}\label{eq:equiv formula4}
  \begin{split}
   A_\al & \rar{\hyperref[fig:type4]{\operatorname{(IV)}}} V_\al\o V_\al' \\
1 &\mapsto \v_0 \o \v_1' + \v_1 \o \v_0'\\
X & \mapsto  \al_0 \v_0 \o \v_1' + \al_1 \v_1 \o \v_0'
  \end{split}
\end{equation}
\end{minipage}
\vskip1em
\noindent
 
Let $\Cs \subset \A$ consist of $m>0$ essential circles, and let $C$ be the $i$-th essential circle in $\Cs$. Consider the cobordism $S$ whose underlying surface is the identity cobordism $\Cs \times I$, with a single dot on the component $C\times I$, as shown in Figure \ref{fig:single dot on identity}. Then $\Ga(S)$ is the identity on all tensor factors except the one corresponding to $C$, and on $C$ it is given by the left-hand side of \eqref{eq:X on odd and even} if $i$ is odd, and the right-hand side if $i$ is even. 
 
 \begin{figure}
\centering
\includegraphics{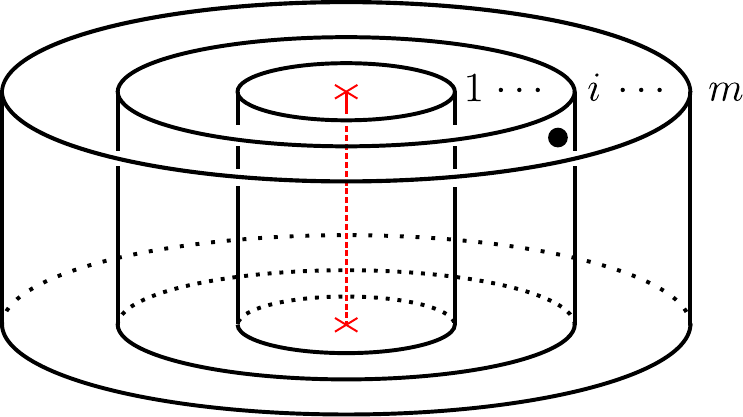}
\caption{Product cobordism on $m>0$ essential circles, with the $i$-th component dotted}\label{fig:single dot on identity}
\end{figure}

\begin{align}\label{eq:X on odd and even} 
\begin{aligned}
   V_\al & \to V_\al \\
\v_0 & \mapsto \al_0 \v_0 \\
\v_1 & \mapsto \al_1 \v_1
\end{aligned}
&&
\begin{aligned}
  V_\al' & \to V_\al' \\
\v_0' & \mapsto \al_1 \v_0' \\
\v_1' & \mapsto \al_0 \v_1'
\end{aligned}
\end{align}
Observe that the functor $\Ga$ is not monoidal, since the action of $X$ on an essential circle depends on its nestedness. 

Let $L\subset \A \times I$ be an oriented link with diagram $D$. Let 
\[
CKh^\A_\al(D) := \Ga([[D]])
\]
denote the chain complex obtained by applying $\Ga$ to the formal complex $[[D]]$. The differential preserves bidegree, and the complex is an invariant of $L$ up to bidegree-preserving chain homotopy equivalence. 

The remainder of this section discusses variants of $\Ga$. Instead of setting both $\al_0 = \al_1=0$, it is possible to set only $\al_0 = 0$ and rename the remaining parameter $\al_1$ to $\al_1 = h$. Denote the resulting Frobenius pair by $(R_h, A_h)$. Explicitly,
\[
R_h = \Z[h],\ A_h = R_h [X]/(X^2-hX).
\]
It may also be obtained from $(R,A)$ by setting $E_1 =h$, $E_2=0$; note that the obstruction in Section \ref{sec:preliminary observation} disappears when $E_2=0$. Collapsing $(R_h, A_h)$ further to characteristic $2$ (that is, applying $(-) \o_{R_h} \Z_2 [h])$ recovers Bar-Natan's theory \cite[Section 9.3]{BN2}. We expect that the resulting annular homology is related to \cite{TW}.

Let $L\subset \A \times I$ be an oriented link with diagram $D$. Viewing $D$ as a diagram in $\R^2$ and applying $\Fa$ to $[[D]]$ yields a chain complex $CKh_\al(D)$ of bigraded $R_\al$-modules. Letting $\d$ denote the differential, Lemma \ref{lem:el cobs split} implies that $\d$ splits as 
\[
\d = \d_0 + \d_2
\]
where $\d_0$ is of bidegree $(0,0)$ and $\d_2$ is of bidegree $(0,2)$. As in \cite{HKLM}, we can introduce an extra parameter $\be$ to account for $\d_2$. Let $R_{\al \be} = R_\al[\be]$ with $\be$ in bidegree $(0,-2)$, and let $CKh_{\al \be}(D)$ be the chain complex over $R_{\al \be}$ with
\[
CKh_{\al \be}^i (D) := CKh_{\al}^i(D) \o_{R_\al} R_{\al \be}
\]
in homological degree $i$ and differential $\d_\be$ given by 
\[
\d_\be := \d_0 + \be \d_2.
\]
Note that $\d_\be$ preserves bidegree. Maps assigned to the four elementary saddles in Figure \ref{fig:el saddles} are given below.

\noindent\begin{minipage}{.5\textwidth}
\begin{equation*}
  \begin{split}
   V_\al \o A_\al & \rar{\hyperref[fig:type1]{\operatorname{(I)}}} V_\al \\  \v_0 \o 1 & \mapsto \v_0 \\ \v_1\o 1 &\mapsto \v_1 \\
    \v_0 \o X & \mapsto \al_0 \v_0 +\ \be v_1 \\
    \v_1 \o X & \mapsto \al_1 \v_1
  \end{split}
\end{equation*}
\end{minipage}%
\begin{minipage}{.38\textwidth}
\begin{equation*}
  \begin{split}
   V_\al \o V_\al' & \rar{\hyperref[fig:type2]{\operatorname{(II)}}} A_\al \\  \v_0 \o \v_0' &\mapsto \be\\
\v_1 \o \v_0' & \mapsto X - \al_0 \\
\v_0 \o \v_1' & \mapsto X - \al_1  \\
\v_1 \o \v_1' & \mapsto 0
  \end{split}
\end{equation*}
\end{minipage}

\vskip1em

\noindent\begin{minipage}{.49\linewidth}
\begin{equation*}
  \begin{split}
   V_\al & \rar{\hyperref[fig:type3]{\operatorname{(III)}}} V_\al \o A_\al \\
\v_0 & \mapsto \v_0 \o X - \al_1 \v_0 \o 1 + \be \v_1 \o 1\\
\v_1 & \mapsto \v_1 \o X - \al_0 \v_1 \o 1
  \end{split}
\end{equation*}
\end{minipage} 
\begin{minipage}{.5\linewidth}
\begin{equation*}
  \begin{split}
   A_\al & \rar{\hyperref[fig:type4]{\operatorname{(IV)}}} V_\al\o V_\al' \\
1 &\mapsto \v_0 \o \v_1' + \v_1 \o \v_0'\\
X & \mapsto  \al_0 \v_0 \o \v_1' + \al_1 \v_1 \o \v_0' + \be \v_1 \o \v_1'
  \end{split}
\end{equation*}
\end{minipage}
\vskip1em
\noindent

\subsection{Inverting $\mathcal{D}$ in equivariant annular homology}\label{sec:inverting D equiv annular}

Recall the Frobenius pair $(R_\ald, A_\ald)$ from \cite{KR}, which was reviewed in Section \ref{sec:inverting D}. Let $\Gad$ denote the composition 
\[
\BN_\al(\A) \rar{\Ga} R_\al \ggmod \to R_\ald \ggmod
\]
where the second functor is extension of scalars. Consider the following elements of $A_\ald$, 
\begin{align*}
\b{v}_0 := \v_0 = 1, \hskip2em \b{v}_1 := \frac{v_1}{\al_1 - \al_0} = \frac{X - \al_0}{\al_1 - \al_0}, \\
\b{v}_0':= \v_0' = 1, \hskip2em \b{v}_1' := \frac{v_1'}{\al_0 - \al_1} = \frac{X - \al_1}{\al_0 - \al_1}.
\end{align*}

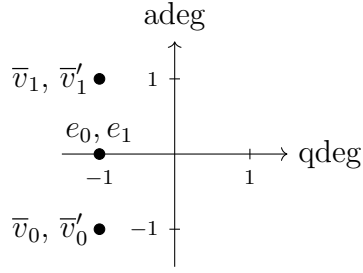
\begin{figure}
\centering
\begin{tikzpicture}
\draw[->] (-1.5,0) -- (1.5,0) node[right] {$\qdeg$};
\foreach \x in { -1, 1}
\draw[shift={(\x,0)}] (0pt,2pt) -- (0pt,-2pt) node[below] {\tiny $\x$};
\draw[->] (0,-1.5) -- (0,1.5) node[above] {$\adeg$};
\foreach \y in {-1,1}
\draw[shift={(0,\y)}] (2pt,0pt) -- (-2pt,0pt) node[left] {\tiny $\y$};
\filldraw[black] (-1,1) circle (2pt) node[anchor=east] {$\b{v}_1$, $\b{v}_1'$};
\filldraw[black] (-1,-1) circle (2pt) node[anchor=east] {$\b{v}_0$, $\b{v}_0'$};
\filldraw[black] (-1,0) circle (2pt) node[anchor=south] {$e_0, e_1$};
\end{tikzpicture}
\caption{Bigradings}\label{fig:bigradings alpha D}
\end{figure}

As in Section \ref{sec:defining the functor}, let $V_\ald$ and $V_\ald'$ denote the module $A_\ald$ with distinguished homogeneous bases $\{\b{v}_0, \b{v}_1\}$ and $\{\b{v}_0', \b{v}_1'\}$, respectively. For a collection of circles $\Cs \subset \A$, the $i$-th essential circle in $\Cs$ is assigned $V_\ald$ if $i$ is odd and $V_\ald'$ if $i$ is even. The notation $A_\ald$ is reserved for trivial circles, with distinguished basis $\{e_0, e_1\}$, see \eqref{eq:e_0 and e_1}. Bigradings are summarized in Figure \ref{fig:bigradings alpha D}. 

With respect to these bases, the maps assigned to the four elementary saddles in Figure \ref{fig:el saddles} are recorded below.

\noindent
\begin{minipage}{.5\textwidth}
\begin{equation}\label{eq:equiv D formula1}
  \begin{split}
   V_\ald \o A_\ald & \rar{\hyperref[fig:type1]{\operatorname{(I)}}} V_\ald \\  
   \b{v}_0 \o e_0 & \mapsto 0 \\ 
   \b{v}_1\o e_0 &\mapsto \b{v}_1 \\
   \b{v}_0 \o e_1 & \mapsto  \b{v}_0 \\
   \b{v}_1 \o e_1 & \mapsto 0
  \end{split}
\end{equation}
\end{minipage}%
\begin{minipage}{.38\textwidth}
\begin{equation}\label{eq:equiv D formula2}
  \begin{split}
   V_\ald \o V_\ald' & \rar{\hyperref[fig:type2]{\operatorname{(II)}}} A_\ald \\ 
    \b{v}_0 \o \b{v}_0' &\mapsto 0 \\
    \b{v}_1 \o \b{v}_0' & \mapsto e_0 \\
    \b{v}_0 \o \b{v}_1' & \mapsto e_1 \\
    \b{v}_1 \o \b{v}_1' & \mapsto 0
  \end{split}
\end{equation}
\end{minipage}

\vskip1em
 
\noindent\begin{minipage}{.49\linewidth}
\begin{equation}\label{eq:equiv D formula3}
  \begin{split}
   V_\ald & \rar{\hyperref[fig:type3]{\operatorname{(III)}}} V_\ald \o A_\ald \\
\b{v}_0 & \mapsto (\al_0 - \al_1) \b{v}_0 \o e_1 \\
\b{v}_1 & \mapsto (\al_1 - \al_0) \b{v}_1 \o e_0 
  \end{split}
\end{equation}
\end{minipage} 
\begin{minipage}{.5\linewidth}
\begin{equation}\label{eq:equiv D formula4}
  \begin{split}
   A_\ald & \rar{\hyperref[fig:type4]{\operatorname{(IV)}}} V_\ald\o V_\ald' \\
e_0 &\mapsto (\al_1 - \al_0) \b{v}_1 \o \b{v}_0'\\
e_1 & \mapsto  (\al_0 - \al_1) \b{v}_0 \o \b{v}_1'
  \end{split}
\end{equation}
\vskip1em
\end{minipage}
To obtain the full set of maps -- that is, if other essential circles are present -- one interchanges $\al_0 \lrar \al_1$, which has the effect of interchanging $\b{v}_0 \lrar \b{v}_0'$, $\b{v}_1 \lrar \b{v}_1'$, and $e_0 \lrar e_1$. They are recorded below for convenience. 

\vskip1em
\noindent
\begin{minipage}{.5\textwidth}
\begin{equation}\label{eq:equiv D formula1'}
  \begin{split}
   V_\ald' \o A_\ald & \to V_\ald' \\  
   \b{v}_0' \o e_0 & \mapsto \b{v}_0' \\ 
   \b{v}_1'\o e_0 &\mapsto 0 \\
   \b{v}_0' \o e_1 & \mapsto  0 \\
   \b{v}_1' \o e_1 & \mapsto \b{v}_1'
  \end{split}
\end{equation}
\end{minipage}%
\begin{minipage}{.38\textwidth}
\begin{equation}\label{eq:equiv D formula2'}
  \begin{split}
   V_\ald' \o V_\ald & \to A_\ald \\ 
    \b{v}_0' \o \b{v}_0 &\mapsto 0 \\
    \b{v}_1' \o \b{v}_0 & \mapsto  e_1 \\
    \b{v}_0' \o \b{v}_1 & \mapsto  e_0 \\
    \b{v}_1' \o \b{v}_1 & \mapsto 0
  \end{split}
\end{equation}
\end{minipage}

\vskip1em
 
\noindent\begin{minipage}{.49\linewidth}
\begin{equation}\label{eq:equiv formula3'}
  \begin{split}
   V_\ald' & \to V_\ald' \o A_\ald \\
\b{v}_0' & \mapsto (\al_1 - \al_0) \b{v}_0' \o e_0 \\
\b{v}_1' & \mapsto (\al_0 - \al_1) \b{v}_1' \o e_1 
  \end{split}
\end{equation}
\end{minipage} 
\begin{minipage}{.5\linewidth}
\begin{equation}\label{eq:equiv formula4'}
  \begin{split}
   A_\ald & \to V_\ald'\o V_\ald \\
e_0 & \mapsto (\al_1 - \al_0) \b{v}_0' \o \b{v}_1\\
e_1 & \mapsto (\al_0 - \al_1) \b{v}_1' \o \b{v}_0
  \end{split}
\end{equation}
\end{minipage}
\vskip1em 

These maps may be written uniformly in the following way. Let $\Cs \subset \A$ be a collection of circles, and label each circle in $\Cs$ by one of the letters $\ab$ or $\bb$. From such a labeling we obtain a distinguished basis element of $\Gad(\Cs)$ by using the correspondence
\begin{equation}\label{eq:a b on trivial}
\ab \lrar e_0,\ \bb \lrar e_1
\end{equation}
for a trivial circle, and 
\begin{equation}\label{eq:a b on essential}
\ab \lrar 
\begin{cases}
\b{v}_1 & i \text{ is odd} \\
\b{v}_0' & i \text{ is even}
\end{cases},\ 
\bb \lrar 
\begin{cases}
\b{v}_0 & i \text{ is odd} \\
\b{v}_1' & i \text{ is even}
\end{cases}
\end{equation}
on the $i$-th essential circle. Then the saddle maps are 

\noindent
\begin{minipage}{.5\textwidth}
\begin{equation}\label{eq:equiv D formula1 ab}
  \begin{split}
   V_\ald \o A_\ald & \rar{\hyperref[fig:type1]{\operatorname{(I)}}} V_\ald \\  
   \bb \o \ab & \mapsto 0 \\ 
   \ab \o \ab &\mapsto \ab \\
   \bb \o \bb & \mapsto  \bb \\
   \ab \o \bb & \mapsto 0
  \end{split}
\end{equation}
\end{minipage}%
\begin{minipage}{.38\textwidth}
\begin{equation}\label{eq:equiv D formula2 ab}
  \begin{split}
   V_\ald \o V_\ald' & \rar{\hyperref[fig:type2]{\operatorname{(II)}}} A_\ald \\ 
    \bb \o \ab &\mapsto 0 \\
    \ab \o \ab & \mapsto  \ab \\
    \bb \o \bb & \mapsto \bb \\
    \ab \o \bb & \mapsto 0
  \end{split}
\end{equation}
\end{minipage}

\vskip1em
 
\noindent\begin{minipage}{.49\linewidth}
\begin{equation}\label{eq:equiv D formula3 ab}
  \begin{split}
   V_\ald & \rar{\hyperref[fig:type3]{\operatorname{(III)}}} V_\ald \o A_\ald \\
\bb & \mapsto (\al_0 - \al_1) \bb \o \bb \\
\ab & \mapsto (\al_1 - \al_0) \ab \o \ab
  \end{split}
\end{equation}
\end{minipage} 
\begin{minipage}{.5\linewidth}
\begin{equation}\label{eq:equiv D formula4 ab}
  \begin{split}
   A_\ald & \rar{\hyperref[fig:type4]{\operatorname{(IV)}}} V_\ald\o V_\ald' \\
\ab &\mapsto (\al_1 - \al_0) \ab \o \ab\\
\bb & \mapsto  (\al_0 - \al_1) \bb \o \bb
  \end{split}
\end{equation}
\vskip1em
\end{minipage}
Moreover, the same formulas hold with $V_\ald$ and $V_\ald'$ interchanged. 

For an annular link $L$ with diagram $D$, let 
\[
CKh^\A_{\ald}(D) := \Gad([[D]])
\]
denote the chain complex obtained by applying $\Gad$ to $[[D]]$. It is an invariant of $L$ up to chain homotopy equivalence, so we may write $Kh^\A_{\ald}(L)$ to denote the homology of $CKh^\A_{\ald}(D)$, for any diagram $D$ of $L$.

\begin{theorem}
Let $L \subset \A \times I$ be a link with diagram $D$. Viewing $L$ as a link in $\R^3$, there is a $\qdeg$-preserving isomorphism 
\[
\varphi\: CKh^\A_{\ald}(D) \rar{\sim} CKh_{\ald}(D). 
\]
\end{theorem}
\begin{proof}
For a smoothing $D_u$, the inclusion $\A \hookrightarrow \R^2$ induces an isomorphism
\[
\varphi_u\: \Gad(D_u) \to \Fad(D_u),
\]
defined in terms of the basis elements labeled by $\ab$ and $\bb$ by 
\[
\ab \mapsto e_0,\ \bb\mapsto e_1.
\]
Comparing the formulas \eqref{eq:equiv D formula1 ab}--\eqref{eq:equiv D formula4 ab} with multiplication and comultiplication in $A_\ald$, we see that each of the maps $\varphi_u$ commute with cobordism maps and thus assemble into an isomorphism $\varphi\: CKh^\A_{\ald}(D) \to CKh_{\ald}(D)$. It is evident from Figure \ref{fig:bigradings alpha D} that each $\varphi_u$ preserves $\qdeg$. Quantum grading shifts in both chain complexes are the same, so the isomorphism $\varphi$ preserves $\qdeg$ as well.  
\end{proof}

The following is immediate from Proposition \ref{prop:Lee rank}. 
\begin{corollary}
For a link $L\subset \A\times I$ with $k$ components, the homology $Kh^\A_{\ald}(L)$ is a free $R_\ald$-module of rank $2^k$. 
\end{corollary}

We recall the \emph{canonical generators} for Lee homology, following \cite{Le} and \cite{Ra}. Let $L\subset \A \times I$ be a link with diagram $D$. Given an orientation $o$ on $L$, let $D_o\subset \A$ denote the result of performing the oriented resolution at each crossing,
\begin{equation*}\label{fig:oriented resolution}
\begin{gathered}
\includegraphics{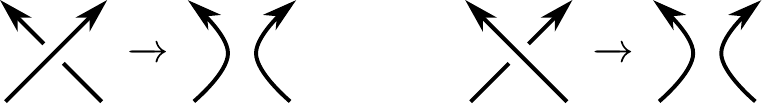}
\end{gathered}\, .
\end{equation*}
Each of the resulting circles is naturally oriented. Assign a mod $2$ invariant to each circle $C$ as follows. First, consider the number of circles in $D_o$ separating $C$ from infinity, mod $2$. Add $1$ if $C$ is standardly (counterclockwise) oriented, and add $0$ otherwise. Now that each circle in $D_o$ is labeled by $0$ or $1$, use the correspondence $0 \lrar \ab$, $1\lrar \bb$ to label each circle by $\ab$ or $\bb$, and finally use \eqref{eq:a b on trivial} and \eqref{eq:a b on essential} to obtain a generator $\s_o$ in $CKh^\A_{\ald}(D)$. 

For a collection of oriented circles $\Cs\subset \A$, let $w(\Cs)$ denote the winding number of $\Cs$. That is, $w(\Cs)$ equals the number of counterclockwise essential circles minus the number of clockwise ones. If $C_1,\ldots , C_m$ are the essential circles in $\Cs$, then 
\[
w(\Cs) = \sum_{i=1}^m w(C_i). 
\]

\begin{proposition}\label{prop:adeg of canonical generators}
Let $L\subset \A \times I$ be a link with diagram $D$, and let $o$ be an orientation of $L$. Let $m$ be the number of essential circles in the oriented resolution $D_o$. Then 
\[
\adeg(\s_o) = (-1)^m w(L,o)
\]
where $w(L,o)$ is the winding number of $L$ with respect to the orientation $o$.  
\end{proposition}

\begin{proof}
First note that $w(L,o) = w(D_o)$. It is straightforward to verify that each essential circle $C$ in $D_o$ contributes $(-1)^m w(C)$ to the annular degree of $\s_o$. The claim follows, since trivial circles do not contribute to the annular degree or the winding number.
\end{proof}

\section{Dotted Temperley-Lieb category}\label{sec:dotted TL}

This section reviews the Temperley-Lieb category $\TL$ and its relation to annular Khovanov homology, following observations in \cite{BPW}. We then introduce a natural equivariant analogue. 

For each $n\geq 0$, fix a collection $\u{n} \subset (0,1)$ of $n$ points in the interior of the unit interval. A \emph{planar} $(n,m)$\emph{-tangle} is a smooth embedding of a compact $1$-manifold $M$ into $I\times I$, such that the boundary of $M$ maps to $I \times \{0\} \cup I \times \{1\}$, with $n$ points in $\d M$ mapping to $\u{n} \subset I\times \{0\}$, and the remaining $m$ points in $\d M$ mapping to $\u{m} \subset \times \{1\}$. Planar tangles are always considered up to isotopy of $I\times I$ fixing the boundary. 

Let $\TL$ denote the $\Zq$-linear category whose objects are nonnegative integers. The morphism space $\TL(n,m)$ is freely generated over $\Zq$ by planar $(n,m)$-tangles, modulo the local relation that an innermost circle can be removed at the cost of multiplying the remaining tangle by $q+q^{-1}$, 
\begin{equation}\label{eq:TL rel}
\begin{gathered}
\includegraphics{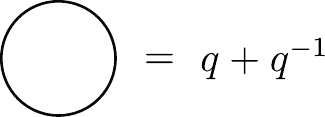}
\end{gathered}\, .
\end{equation}
Composition is defined by vertically stacking planar tangles. Denote the space of morphisms from $n$ to $m$ by
\[
\TL(n,m).
\]
The Temperley-Lieb algebra $TL_n$ can then be identified with the endomorphism space $\TL(n,n)$. Let $\TL_{q=1}$ denote the category obtained by setting $q=1$.

It was observed in \cite[Section 6.1]{BPW} that $\TL$ is closely tied to annular Khovanov homology. By spinning planar tangles in the $S^1$ direction, one obtains a functor
\begin{equation}\label{eq:spinning 1}
S^1 \times (-) \: \TL_{q=1} \to \BN(\A).
\end{equation}
Explicitly, $n$ is sent to the essential circles $S^1\times \u{n} \subset \A$, and a planar tangle $T$ is sent to the cobordism $S^1\times T$. That $S^1\times (-)$ factors through the relation \eqref{eq:TL rel} when $q=1$ follows from the fact that a torus in $\A\times I$ evaluates to $2$ in $\BN(\A)$. Let 
\[
\BBN(\A)
\]
denote the category obtained from $\BN(\A)$ by imposing Boerner's relation, Figure \ref{fig:Boerner}. Recall that the non-equivariant annular TQFT $\Gn$ factors through this relation, so no information is lost from the point of view of link homology. 

It is shown in \cite[Section 4.2]{GLW} that $\Gn$ can be made to take values in the representation category of $\sl_2$. On the other hand, it is well-known that Temperley-Lieb diagrams (planar tangles modulo relation \eqref{eq:TL rel}) can be interpreted as $U_q(\sl_2)$-linear maps between tensor powers of the fundamental representation of $U_q(\sl_2)$; a convenient reference is \cite[Appendix A.1]{BPW}. Thus for $q=1$, both $\TL_{q=1}$ and $\BBN(\A)$ admit functors $\mathcal{F}_{\TL}$ and $\Gn$ to $U(\sl_2)\mmod$. It was observed in \cite[(6.1)]{BPW} that the following diagram commutes.

\begin{equation}\label{eq:commuting triangle}
\begin{tikzcd}[column sep=tiny]
\TL_{q=1} \arrow[rr, "S^1\times (-)"] \arrow[rd, "\mathcal{F}_{\TL}"'] & & \BBN(\A) \arrow[ld, "\Gn"] \\
 & U(\sl_2)\mmod & 
\end{tikzcd}
\end{equation}
Moreover, if $\TL_{q=1}$ is made additive and graded by introducing formal direct sums \cite[Definition 3.2]{BN2} and formal grading shifts \cite[Section 6]{BN2}, then the horizontal functor in \eqref{eq:commuting triangle} becomes an equivalence of categories, \cite[Proposition 6.1]{BPW}.

Thus $\TL_{q=1}$ characterizes the skein category $\BBN(\A)$ and the functor $\F_{\TL}$ characterizes the annular TQFT $\Gn$. On the other hand, $\BN(\A)$ was similarly described using planar diagrams in \cite{Ru}, which we restate now. A \emph{dotted} planar $(n,m)$-tangle is a planar $(n,m)$ tangle whose components may be decorated by some number of dots, which are allowed to float freely along a component. Let 
\[
\TL_\bullet
\]
denote the category whose objects are nonnegative integers and whose morphisms are $\Z$-linear combinations of dotted planar tangles, modulo the additional local relations in Figure \ref{fig:TLd rels}. 

\begin{figure}
\begin{subfigure}[b]{.3\textwidth}
\begin{center}
\includegraphics{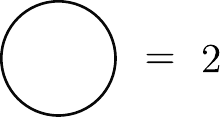}
\end{center}
\caption{Undotted circle}\label{fig:undotted circle TLd}
\end{subfigure}
\begin{subfigure}[b]{.2\textwidth}
\begin{center}
\includegraphics{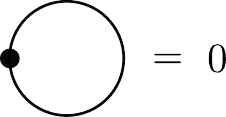}
\end{center}
\caption{Dotted circle}\label{fig:dotted circle TLd}
\end{subfigure}
\begin{subfigure}[b]{.4\textwidth}
\begin{center}
\includegraphics{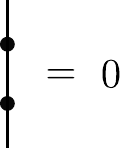}
\end{center}
\caption{Two dots relation}\label{fig:two dots TLd}
\end{subfigure}
\caption{Relations in $\TL_\bullet$.}\label{fig:TLd rels}
\end{figure}

For a dotted planar  tangle $T$, let $T^u$ denote the tangle obtained by removing all dots ($u$ stands for undotted). Consider the functor 
\begin{equation}\label{eq:spinning d}
S^1 \times (-) \: \TL_\bullet \to \BN(\A)
\end{equation} which sends a dotted planar tangle $T$ to the cobordism whose underlying surface is $S^1 \times T^u$, and which carries $k$ dots on a component if $T$ carried $k$ dots on the corresponding component. The relations in Figure \ref{fig:TLd rels} are planar analogues of relations in $\BN(\A)$, Figure \ref{fig:BN relations}. Figures \ref{fig:undotted circle TLd} and \ref{fig:dotted circle TLd} correspond to an undotted torus and a once-dotted torus evaluating to $2$ and $0$ in $\BN(\A)$, respectively. Figure \ref{fig:two dots TLd} corresponds to the two dots relation in Figure \ref{fig:two dot}. 

Upon introducing formal direct sums and formal grading shifts, the argument in \cite[Proposition 6.1]{BPW} shows that the functor \eqref{eq:spinning d} is essentially surjective and full. It is not faithful, but by \cite[Theorem 3.1]{Ru}, its kernel is generated by the local relations shown in Figure \ref{fig:Russel relations nonequiv}. Note that the second follows from the first by adding a dot near one of the endpoints of the strands and simplifying using the two dots relation. 
\begin{figure}
\includegraphics{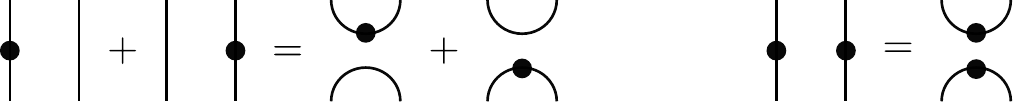}
\caption{Russel relations.}\label{fig:Russel relations nonequiv}
\end{figure}
To see that the relations hold, consider two annuli embedded in $\A \times I$ with a tube joining them, Figure \ref{fig:tubed annuli}, and perform neck-cutting along the two disks shown in Figure \ref{fig:tubed annuli2} and Figure \ref{fig:tubed annuli1}. Denote by 
\[
\til{\TL}_\bullet
\]
the category obtained by imposing the Russel relations. It follows from \cite[Theorem 3.1]{Ru} that the induced functor 
\[
S^1 \times (-) \: \widetilde{\TL}_\bullet \to \BN(\A)
\]
becomes an equivalence of categories after introducing formal direct sums and formal grading shifts to $\widetilde{\TL}_\bullet$.

\begin{figure}
\begin{subfigure}[b]{.3\textwidth}
\begin{center}
\includegraphics{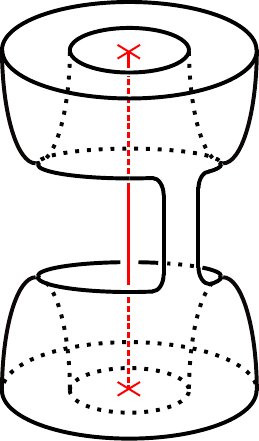}
\end{center}
\caption{}\label{fig:tubed annuli}
\end{subfigure}
\begin{subfigure}[b]{.3\textwidth}
\begin{center}
\includegraphics{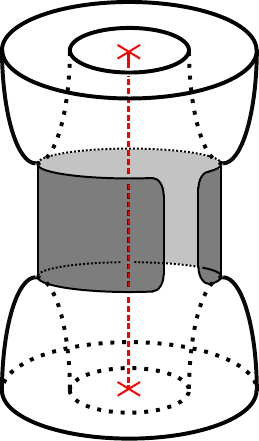}
\end{center}
\caption{}\label{fig:tubed annuli2}
\end{subfigure}
\begin{subfigure}[b]{.3\textwidth}
\begin{center}
\includegraphics{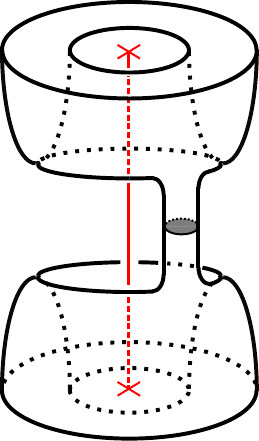}
\end{center}
\caption{}\label{fig:tubed annuli1}
\end{subfigure}
\caption{Two compression disks.}\label{fig:Russel rels explained}
\end{figure}

An equivariant version of $\TL_\bullet$ follows from considering the skein category $\BN_\al(\A)$. Arguing as in \cite[Proposition 6.1]{BPW}, any object of $\BN_\al(\A)$ is isomorphic to a direct sum of grading-shifted essential circles. Any cobordism in $\BN_\al(\A)$ can be expressed, in a non-unique way, as an $R_\al$-linear combination of cobordisms of the form $S^1\times T$, where $T$ is a dotted planar tangle. It follows that any additive functor out of $\BN_\al(\A)$ is determined by its value on each collection of $n\geq 0$ essential circles and on cobordisms of the form $S^1\times T$. This naturally leads to the following definition.

\begin{definition}
Let $\TL_\al$ denote the category whose objects are nonnegative integers, and whose morphisms are formal $R_\al$-linear combinations of dotted planar tangles, modulo the local relations shown in Figure \ref{fig:TLa rels}. 
\end{definition}

For a dotted planar tangle $T$ with $d(T)$ dots, define its degree to be 
\begin{equation*}\label{eq:deg tangle}
\deg(T) = 2 d(T).
\end{equation*}
Note that the relations in Figure \ref{fig:TLa rels} are homogeneous.

\begin{figure}
\begin{subfigure}[b]{.25\textwidth}
\begin{center}
\includegraphics{AD_rel_TL1.pdf}
\end{center}
\caption{Undotted circle}\label{fig:undotted circle}
\end{subfigure}
\begin{subfigure}[b]{.25\textwidth}
\begin{center}
\includegraphics{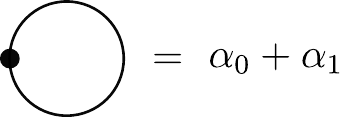}
\end{center}
\caption{Dotted circle}\label{fig:dotted circle}
\end{subfigure}
\begin{subfigure}[b]{.4\textwidth}
\begin{center}
\includegraphics{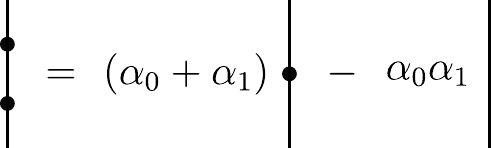}
\end{center}
\caption{Two dots relation}\label{fig:two dots TLa}
\end{subfigure}
\caption{Relations in $\TL_\al$.}\label{fig:TLa rels}
\end{figure}

To motivate the relations, consider the functor 
\begin{equation}\label{eq:spinning alpha}
S^1 \times (-) \: \TL_\al \to \BN_\al(\A)
\end{equation}
defined as in \eqref{eq:spinning d}. An undotted torus and a once-dotted torus in $\A \times I$ evaluate to $2$ and $\al_0 + \al_1$ in $\BN_\al(\A)$, respectively, which explains the relations in Figure \ref{fig:undotted circle} and Figure \ref{fig:dotted circle}. The relation in Figure \ref{fig:two dots TLa} is a planar analogue of the two dots relation in $\BN_\al(\A)$, see Figure \ref{fig:alpha two dot}. A straightforward induction argument also shows that an innermost circle with $k\geq 0$ dots evaluates to $\al_0^k + \al_1^k$ in $\TL_\al$,
\begin{equation}\label{eq:k dotted circle}
\begin{gathered}
\includegraphics{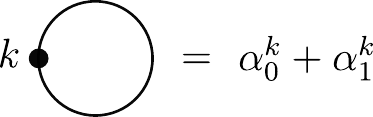}\, .
\end{gathered}
\end{equation}
By composing \eqref{eq:spinning alpha} with the equivariant annular TQFT $\Ga$, 
\begin{equation}\label{eq:TLa as linear maps}
\TL_\al \rar{S^1 \times (-)} \BN_\al(\A) \rar{\Ga} R_\al\ggmod,
\end{equation}
one can view dotted planar tangles as $R_\al$-linear maps between tensor powers of $A_\al$. 

\begin{remark}\label{rmk:U(2) vs U(1) times U(1)}
As is the case for $\BN_\al(\A)$, the relations in $\TL_\al$ involve only symmetric polynomials in $\al_0$ and $\al_1$, so one may consider the $U(2)$-equivariant analogue instead; see also Remark \ref{rmk:cob cat BNa is induced from U(2)}. However, the functor \eqref{eq:TLa as linear maps} is not present in the $U(2)$-equivariant setting. 
\end{remark}

The functor \eqref{eq:spinning alpha} is of course not faithful. For example, it factors through the local relations shown in Figure \ref{fig:Russel relations}, which are equivariant analogues of the Russel relations, Figure \ref{fig:Russel relations nonequiv}. Let
\[
\widetilde{\TL}_{\al}
\]
denote the category obtained from $\TL_\al$ by imposing the relations in Figure \ref{fig:Russel relations} (it suffices to impose only the first). 

\begin{figure}
\centering
\includegraphics{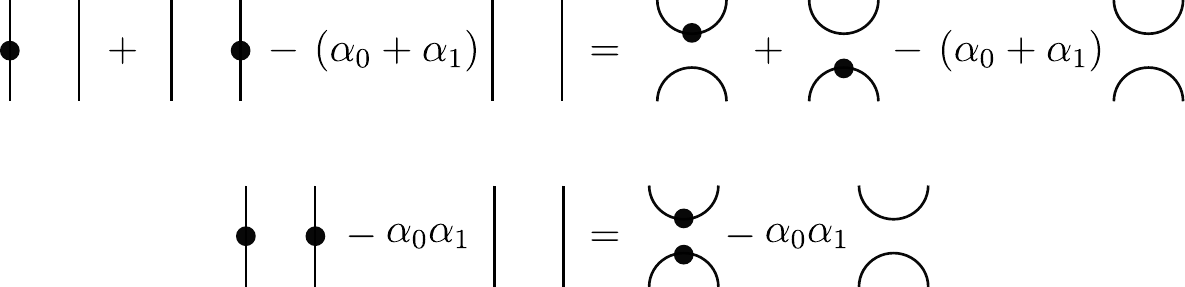}

\caption{Equivariant Russel relations.}\label{fig:Russel relations}
\end{figure}
We end the section with two questions. The first is motivated by \cite[Theorem 3.1]{Ru}. 
\begin{question}{1}
Is the induced spinning functor $S^1\times (-)\: \widetilde{\TL}_\al \to \BN_\al(\A)$ faithful? 
\end{question}

By \cite[Main Theorem]{Ru} and results in \cite{Kh2}, the abelian group $\widetilde{\TL}_\bullet(2n,0)$ is free of rank 
\[
\binom{2n}{n}.
\]Note that for a symbol $\star \in \{\varnothing, \bullet, \alpha\}$, the modules $\TL_\star(n,m)$ and $\TL_\star(k,\l)$ are isomorphic whenever $n+m = k + \l$. An isomorphism $\TL_\star(n,m) \to \TL_\star(n+m,0)$ and its inverse are depicted in Figure \ref{fig:tangle maps}. It clearly factors through the Russel relations, so $\widetilde{\TL}_\bullet(n,m)$ is free of rank 
\begin{equation}\label{eq:rank of TLd}
\binom{n+m}{(n+m)/2}
\end{equation}
whenever $n$ and $m$ have the same parity, and otherwise it is the zero module. 

\begin{question}{2}
Is $\til{\TL}_\al(n,m)$ free over $R_\al$, and, if so, of what rank?
\end{question}

\begin{figure}
\begin{subfigure}[b]{.4\textwidth}
\begin{center}
\includegraphics{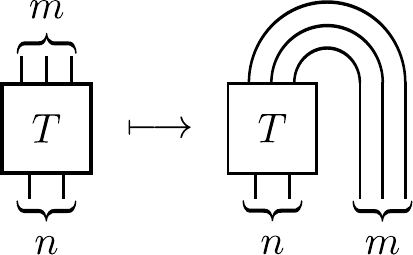}
\end{center}
\caption{$\TL_\star(n,m) \to \TL_\star(n+m,0)$}\label{fig:tangle map1}
\end{subfigure}
\begin{subfigure}[b]{.4\textwidth}
\begin{center}
\includegraphics{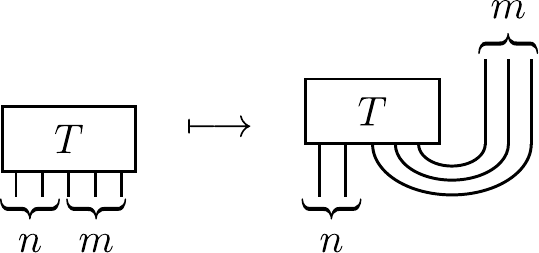}
\end{center}
\caption{$\TL_\star(n+m,0) \to \TL_\star(n,m)$}\label{fig:tangle map2}
\end{subfigure}
\caption{An isomorphism $\TL_\star(n,m) \cong \TL_\star(n+m,0)$.}\label{fig:tangle maps}
\end{figure}

Note that $\TL_\al(n,m) = 0$ if $n$ and $m$ have different parities, and otherwise $\TL_\al(n,m)$ is free of rank 
\[
2^{\l} C(\l),
\]
where $\l = \frac{n+m}{2}$ and $C(\l)$ is the $\l$-th Catalan number. To see this, consider the collection of dotted planar $(n,m)$-tangles in which every component carries at most one dot and which has no closed components. They evidently form a basis for $\TL_\al(n,m)$. There are $C(\l)$ undotted planar $(n,m)$-tangles with no closed components. A fixed such tangle has $\l$ components, hence $2^\l$ ways to put at most one dot on each component, which yields the count.

We can find bases for small values. A basis for $\widetilde{\TL}_{\al}(1,1)$ is given by an undotted and once-dotted vertical strand, and a basis for $\widetilde{\TL}_{\al}(2,2)$ is depicted in \eqref{eq:TLa(2,2) basis}. That these elements generate the module follows from Figure \ref{fig:Russel relations}, and linear independence can be verified using $S^1\times (-)$ and the TQFT $\Ga$. The ranks agree with \eqref{eq:rank of TLd}, but it is not clear if this is a coincidence for small examples. For $\widetilde{\TL}_\alpha(3,3)$, there are $2^3 C(3) = 40$ generators and many relations, making direct computation difficult. 

\begin{equation}\label{eq:TLa(2,2) basis}
\begin{gathered}
\includegraphics{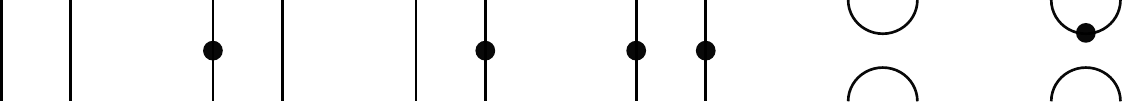}
\end{gathered}
\end{equation}

\end{document}